\documentclass[a4paper 12pt]{article}
\usepackage{latexsym,amsmath, amsthm, amssymb}
\newtheorem{lem}{Lemma}[section]
\newtheorem{thm}{Theorem}[section]
\newtheorem{definition}{Definition}[section]
\newtheorem{corollary}{Corollary}[section]

\newtheorem{example}{Example}[section]
\newtheorem{pro}{Proposition}[section]
\def\bc{\begin{center}}
\def\ec{\end{center}}

\begin{document}
 \abovedisplayskip=8pt plus 1pt minus 1pt
\belowdisplayskip=8pt plus 1pt minus 1pt
%-------------------  First Head  -----------------------------------------
\thispagestyle{empty} \vspace*{-3.0truecm} \noindent
\parbox[t]{6truecm}{\footnotesize\baselineskip=11pt\noindent  {} %Acta Mathematica
%%Sinica, English Series\\
%%1999, Jan., Vol.15, No.1, p. 1--11\\
%%Http://www.ActaMath.com\\
%DOI:
 } \hfill
%%\parbox[c]{6truecm}{\vbox{\hsize 3.6576 true cm %
%%  \vskip 3.8 true cm %1.8373
%%  \relax\hbox to0.4\hsize{\hbox to0pt{\special{BMF=actmark.BMF}}\hss}\hss}}
%\hbox to\textwidth{\vbox{\footnotesize\baselineskip=11pt\noindent
%Acta Mathematica Sinica, English Series\hfill LOGO\\
%1999, Jan., Vol.15, No.1, pp. 1--11\hfill \copyright Spring-Verlag 1999}
%\vbox{\hsize3.6576 true cm
%  \vskip1.8373 true cm
%  \relax\hbox to\hsize{\hbox to0pt{\special{BMF=ACTMARK.BMF}}\hss}\hss}}
%===================Text=============================================
\vspace{1 true cm}

\bc{\large\bf Non-smooth analysis method in optimal investment- a BSDE approach
\footnote{The work of Helin Wu is  Supported by the Scientific and Technological Research
Program of Chongqing Municipal Education Commission  (KJ1400922). The work of Yong Ren is supported by the National
Natural Science Foundation of
 China (11201004 and 11371029).}}
%\footnotetext{\footnotesize Received February 24, 1998, Revised
%September 1, 1998, Accepted September 9, 1996}}
 \ec
\vspace*{0.1 true cm}
\bc{\bf Helin Wu$^{+}$  \\
{\small\it School of Mathematics, Chongqing  University of Technology, Chongqing 400054, China\\
\small\it \quad  Corresponding author. E-mail: wuhelin@cqut.edu.cn}}\ec \vspace*{1 true mm}
\bc{\bf Yong Ren\\
{\small\it Department of Mathematics, Anhui Normal University, Wuhu 24100, China\\
\small\it \quad E-mail: renyong@126.com and brightry@hotmail.com}}\ec \vspace*{3 true mm}

\begin{abstract}
   In this paper, our aim is to investigate necessary conditions for optimal investment. We model the wealth process by Backward differential stochastic equations (shortly for BSDE) with or without constraints on wealth and portfolio process. The constraints can be very general thanks the  non-smooth analysis method we adopted.
\end{abstract}

\section{Introduction}

In the sequel, let $(\Omega,{\cal F},P)$ be a probability space equipped with a
standard Brownian motion $W$. For a fixed real number $T>0$, we
consider the filtration $\mathbb{F}:=({\cal F}_t)_{0\leq t\leq T}$
which is generated by $W$ and augmented by all $P$-null sets. The
filtered probability space $(\Omega,{\cal F},\mathbb{F},P)$
satisfies the usual conditions.

Given an initial capital $x$, the  investor's investment process   is said to vary against  some kind of Backward  Differential Stochastic Equations (short for BSDE),
\begin{equation}\label{wealthprocess}
y_t=\xi+\int_t^Tg(s,y_s,z_s)ds-\int_t^Tz^*_sdW_s, 0\leq t \leq T
\end{equation}
with $y_0\leq x$,  where $g(\omega,t,y,z):\Omega\times[0,T]\times R \times R^d\rightarrow R$ is a function
satisfying  uniformly Lipschitz condition, i.e., there exists a positive constant $M$ such that
for all $(y_1,z_1),(y_2,z_2)\in  R \times R^d$
$$|g(\omega,t,y_1,z_1)-g(\omega,t,y_2,z_2)|\leq M(|y_1-y_2|+|z_1-z_2|)
\eqno(A1)$$ and
$$g(\cdot,0,0)\in H_T^2(R),  \eqno(A2)$$
where $H_T^2(R^d)$ denotes  the space of predictable process
$\varphi:\Omega\times[0,T]\rightarrow R^d$ satisfying
$\parallel \varphi\parallel^2=E\int_0^T|\varphi(s)|^2ds<+\infty.$
We  call $y_t$ wealth process and $z_t$ portfolio process.

The BSDE approach is a backward view  for investment. For comparison, if we take a forward view for the equation \eqref{wealthprocess},
then the complicated  process $z_t$ acts as a control. However, by the theory of BSDE,   $z_t$ is determined by the terminal value $\xi$ via a one-one correspondence. Thus the BSDE approach has
the virtue to handle similar control problem taking $\xi$ as a control instead. Moreover, if we consider the function $y_0:=\mathcal{E}_{t,T}^g(\xi)$ induced by
BSDE with terminal value $\xi$, a terminal perturbation method,  which was  first used  in  Bielecki et al.
\cite{tr-bi&h-ji&sr-pl&xy-zh}, can be used to analyze the optimal investment problem. Along with this line, later in many years, Ji and Peng
\cite{sl-ji&sg-pe} used it to obtain a necessary  condition via Ekeland variation principle. In this paper, as a generalization, we study  optimal investment problems by  non-smooth analysis method, which makes more general optimal problems inside our consideration.

Supposing that the investor has initial wealth $x$, he invest it in the financial market according to the equation \eqref{wealthprocess}.
 By the above analysis, his investment strategy is determined by  all available terminal value  for him.
Let
$$\mathcal{A}(x):=\{\xi\in L_T^2(R)|y_0\leq x\},$$
then our problem is
\begin{equation}\label{minmizerisk}
\min_{\xi\in \mathcal{A}(x)}\rho(\xi),
\end{equation}
where $\rho(\cdot)$  is a function defined on $L_T^2(R)$ which  usually represents a risk measure. However in our paper, it can
be a general Lipschitz function.

Sometimes, one ask $(y_t,z_t)$ satisfy some constraint condition 
$$(y(t),z(t))\in \Gamma_t,\quad a.e.,a.s. \quad \text{on} \quad [0,T]\times \Omega, \eqno(C)$$
where $\Gamma_t:=\{(y,z)|\phi(t,y,z)=0\}\subset R\times R^d$ and  $\phi$ satisfies  conditions (A1) and (A2).
In such constrained case, the investment  model should be changed to a Constrained Backward Differential Equation (shortly for CBSDE) as follows,
\begin{equation}\label{cwealthprocess}
y_t=\xi+\int_t^Tg(s,y_s,z_s)ds+C_T-C_t-\int_t^Tz^*_sdW_s, 0\leq t \leq T,
\end{equation}
 where $C_t$ is an increasing RCLL
(right continuous an left limit exists) process with $C_0= 0$,
$y_t$ is often called a super-solution of BSDE in the literature.

The idea of constrained investment comes from  incompleteness  or other constraints on investment in financial market.
In such case, super-hedging strategies are often adopted. Corresponding to such strategies, the minimal super-solution defined as follow is meaningful.
\begin{definition}\rm ($g_\Gamma$-solution) A $g$-super-solution $(y_t, z_t, C_t)$ is said to be
 the minimal solution, given $y_T=\xi$,
subjected to the constraint $(C)$ if for any other
g-super-solution $(y'_t, z'_t, C'_t)$ satisfying $(C)$  with
$y'_T=\xi$, we have $y_t\leq y'_t $ a.e., a.s.. We call the  minimal
solution  $g_\Gamma$-solution and denote it as
$y_t:=\mathcal{E}_{t,T}^{g,\phi}(\xi)$. In no constrained case, i.e. when $\phi(t,y,z)\equiv0\,\, $$P$-a.s. for any $t\in [0,T]$, we denote it as $\mathcal{E}_{t,T}^g(\xi)$
 for convenience.
\end{definition}

Our problem in the constrained case is similar  to \eqref{minmizerisk} but change $\mathcal{A}(x):=\{\xi\in L_T^2(R)|y_0\leq x\}$ to $\mathcal{A}^\phi(x):=\{\xi\in L_T^2(R)|\mathcal{E}_{t,T}^{g,\phi}(\xi)\leq x\}$, that is to minimize
\begin{equation}\label{cminmizerisk}
\min_{\xi\in \mathcal{A}^\phi(x)}\rho(\xi).
\end{equation}

Our paper is organized as follows. In section 2, we first study optimal investment problem without constraints  on wealth and portfolio process. With the help of non-smooth analysis, we obtain a necessary condition for an optimal solution, which generalize those obtained in Ji and Peng
\cite{sl-ji&sg-pe}. Secondly, we continue to consider constrained case. We point out serious difficulties  we met in this case and discussed such problems briefly, more details and fully discussion about such constrained  problem will be included in our future papers. In sections 3, we give some examples to verify our analysis. At last section, some necessary backgrounds about non-smooth analysis are gathered.

\section{Maximum principle for the optimal investment problem}
In this section, we aim to derive some necessary conditions for the optimality of our problem. Suppose the wealth process of an investor evolving according to
\eqref{wealthprocess} with  limited initial  capital $x$. The optimal problem \eqref{minmizerisk} is a constrained problem.
Just as usual,

Suppose no constraints on wealth and portfolio process,
by an exact penalization method used in non-smooth analysis, we need to assume  that\\
i) \quad The risk measure $\rho(\cdot)$ is Lipschitz\\
ii) \quad If we write $y_0\triangleq \mathcal{E}_{0,T}^g(\cdot)$ as a function of terminal value, it is Lipschitz.\\
If no constraints on wealth and portfolio process, by the theory of BSDE, $\mathcal{E}_{0,T}^g(\cdot)$ is obviously Lipschitz, see E. Pardoux, S.G. Peng \cite{e-pa&sg-pe} for example.
\begin{pro}\label{bsdelip}
Suppose $g$ satisfies conditions (A1) and (A2), $\xi_i\in L^2_T(R)$, $(y^i_t,z^i_t),\,i=1,2$ are solutions of  \eqref{wealthprocess} with terminal values $\xi_i$, then there exists a constant $C>0$ such that
$$|y^2_0-y^1_0|^2\leq CE|\xi_2-\xi_1|^2$$.
\end{pro}

In order to use results  \eqref{lemcalofnor} in appendix, we need to show
\begin{lem}\label{0notbelong}
Supposing that $g$ satisfies conditions {\rm(A1)} and {\rm(A2)}, then for any $\xi^*\in L_T^2(R)$, we have $0\notin \partial^o \mathcal{E}_{0,T}^g(\xi^*)$.
\end{lem}
The proof is very similar to the proof of the strict comparison theorem of BSDE.
\begin{proof}
Suppose on the contrary $0\in \partial^o \mathcal{E}_{0,T}^g(\xi^*)$. Let $f(\cdot)=\mathcal{E}_{0,T}^g(\cdot)$, then
$$f^o(\xi^*, \eta)\geq 0$$ holds
for any $\eta\in L_T^2(R)$.

But by  the definition
$$f^o(\xi^*;\eta):=\limsup_{\xi\rightarrow \xi^*,t\downarrow 0}\frac{\mathcal{E}_{0,T}^g(\xi+t\eta)-\mathcal{E}_{0,T}^g(\xi)}{t}$$

$$\frac{\mathcal{E}_{0,T}^g(\xi+t\eta)-\mathcal{E}_{0,T}^g(\xi)}{t}\leq ME_Q[\eta],$$
when $P(\eta\leq 0)= 1, P(\eta<0)>0$,
where $Q$ is an equivalent measure of $P$, $M$ is a positive number.

In the above equation, set $\eta= -1$, then for any $\xi$,
$$\frac{\mathcal{E}_{0,T}^g(\xi+t\eta)-\mathcal{E}_{0,T}^g(\xi)}{t}\leq -M,$$
we get a contradiction with $f^o(\xi^*, \eta)\geq 0$.
\end{proof}

Since   functions   $\mathcal{E}_{0,T}^g(\cdot)\triangleq y_0$ generated by BSDE via  \eqref{wealthprocess}  is Lipschitz, then according to
in appendix,  $\partial^o \mathcal{E}_{0,T}^g(\xi)$ is not empty for any $\xi$, in fact, we have following results.

\begin{thm}
Suppose that $g$ and $ \xi$ are the standard parameters for BSDE, $(y_t,z_t)$ is a solution of BSDE with terminal value $\xi$. Let $f(\xi)=y_0$, then $f(\cdot)$ is Lipschitz
on Hilbert space $L_T^2(R)$, and
\begin{equation}
\partial^o f(\xi)\subset \int_0^T \langle\partial^o g(t,y_t,z_t), (\tilde{y}_t, \tilde{z}_t)\rangle dt-\int_0^T\tilde{z}^*_t dW_t.
\end{equation}
\end{thm}
The meaning of the equation above is that, for any $\eta\in \partial^o f(\xi) $, there exists $(\varphi_t, \psi_t)\in \partial^o g(t,y_t,z_t)$, such that
for any $\zeta \in L_T^2(R)$,
$$\langle\eta, \zeta\rangle =\tilde{y}_0= \zeta +\int_0^T (\varphi_t\tilde{y}_t +  \psi_t\tilde{z}_t)dt-\int_0^T\tilde{z}^*_t dW_t.$$
where $(\tilde{y}_t, \tilde{z}_t)$ is the solution of BSDE generated by
$h(t,y,z)=\varphi_t y +  \psi_tz$ with terminal value $\zeta$.
\begin{proof}
Let $\hat{g}(t, \hat{y}, \hat{z})=g^o(t, y_t,z_t;\hat{y}, \hat{z})$, where $(\hat{y}, \hat{z})\in R\times R^d$,
then by the Definition of generalized directional derivative, $\hat{g}(t, \hat{y}, \hat{z})$ is Lipschitz homogeneous and convex in $(\hat{y}, \hat{z})$ and
\begin{equation}\label{maxequal}
\hat{g}(t, \hat{y}, \hat{z})=\max_{(\varphi_t, \psi_t)\in \partial^o g(t,y_t,z_t)}\langle (\varphi_t, \psi_t),
(\hat{y}, \hat{z})\rangle.
\end{equation}
For $\zeta \in L_T^2(R)$, the BSDE generated by $\hat{g}$ with terminal value  $\zeta$ evolves as follows
\begin{equation}
\hat{y}_t=\zeta+\int_t^T\hat{g}(t, \hat{y}_t, \hat{z}_t) dt-\int_t^T\hat{z}^*_t dW_t, 0\leq t\leq T.
\end{equation}
Let $\hat{f}(\zeta )=\hat{y}_0$, then $\hat{f}$ is  homogeneous and convex.
Let
\begin{equation}
M:=\{\eta\in L_T^2(R)|\langle\eta, \zeta\rangle =\tilde{y}_0,
(\varphi_t, \psi_t)\in \partial^o g(t,y_t,z_t)\},
\end{equation}
where $(\tilde{y}_t, \tilde{z}_t)$ is the solution of following BSDE
\begin{equation}\label{joi-equ}
\tilde{y}_t= \zeta +\int_t^T (\varphi_t\tilde{y}_t +
\psi_t\tilde{z}_t)dt-\int_t^T\tilde{z}^*_t dW_t, 0\leq t\leq T.
\end{equation}
By the comparison theorem of BSDE, for any  $\eta\in M$, one has $\hat{f}(\zeta )\geq \langle\eta, \zeta\rangle$, and by \eqref{maxequal},
$\hat{f}(\zeta )=\max_{\eta\in M} \langle\eta, \zeta\rangle,$
 thus $\partial \hat{f}(0 )= M$.
By now, if we can proof $\partial^o f(\xi)=\partial \hat{f}(0 )$, then the theorem is proved.
But in fact, by the continuous dependence theorem and comparison proposition, $f^o(\xi, \zeta)= \hat{f}(\zeta)$ can be obtained easily.
\end{proof}

Based on the above results in non-smooth analysis, we get a necessary conditions for the optimality of \eqref{minmizerisk}.
\begin{thm}\label{necessary1}
Suppose that $\rho(\cdot)$ is a
  Lipschitz
function. If $\xi^*$ is an optimal solution of \eqref{minmizerisk}
then for some $\lambda$, there exist  $\zeta\in \partial^o \rho(\xi^*)$
and
$\eta\in\partial^o \mathcal{E}_{0,T}^g(\xi^*)$
such that
$$\zeta+ \lambda \eta= 0$$
holds.
\end{thm}
\begin{proof}
By Lemma \ref{0notbelong}, we have
 $0\notin \partial^o \mathcal{E}_{0,T}^g(\xi^*)$ and thus Proposition \ref{lemcalofnor} can be used to deduce that

 $$N_C(\xi^*)\subset \bigcup_{\lambda\geq 0}\lambda\partial^o \mathcal{E}_{0,T}^g(\xi^*).$$
If $\xi^*$ is an optimal solution of \eqref{minmizerisk}  satisfying
$\mathcal{E}_{0,T}^g(\xi^*)= x$,   then by
 the Fermat condition \eqref{necessaryconditon2}, there exists a nonnegative number $\lambda\geq 0$
and some $\zeta\in \partial^o \rho(\xi^*),$ $\eta\in\partial^o
\mathcal{E}_{0,T}^g(\xi^*)$ such that
$$\zeta+ \lambda \eta= 0.$$
Now supposing $\tilde{x}=\mathcal{E}_{0,T}^g(\xi^*)< x$, then we can set
$$\widetilde{C}:=\{\xi\in L_T^2(R)|\mathcal{E}_{0,T}^g(\xi)\leq \tilde{x}\}$$
and solve optimal problem on $\widetilde{C}$. It is easy to see that  $\xi^*$  is optimal on $\widetilde{C}$ if it is optimal on $C$ for $\rho(\cdot)$.
\end{proof}

In Ji and Peng \cite {sl-ji&sg-pe}, they assume that the generator is continuously differentiable with variables, in this special case, we can get a explicit form of   $\mathcal{E}_{0,T}^g(\cdot)$. To do so, we need a notation named strict differentiable for a function in Banach Space and a related theorem.
\begin{definition}[Strict differentiable, Clark\cite{fh-cl}]
A function $f(\cdot)$ defined on Banach space $X$ is called strict differentiable at $x\in X$ if there exists
$x^*\in X^*$ such that
$$\lim_{y\rightarrow x,t\rightarrow 0^+}\frac{f(y+td)-f(y)}{t}=\langle x^*, d\rangle$$
exists in any direction $d\in X$.
\end{definition}
\begin{thm} {\rm(Clark\cite{fh-cl})}
A function $f(\cdot)$ defined on Banach space $X$ is strict differentiable at $x\in
X$ as in the above definition, then  $\partial f(x)=\{x^*\}$.
\end{thm}
Based on the above notations and results, we have
\begin{lem}\label{gen-gra-bsde}
If  $g$ is continuously differentiable in  $(y,z)\in R\times R^d$ with bounded derivatives,
$(y_t=\mathcal{E}_{0,t}^g(\xi),z_t)$ is the solution of BSDE with terminal value  $\xi$, then
$$ \partial^o \mathcal{E}_{0,T}^g(\xi)=  \{q_T  \},  $$
 where $q_T\in L^2_T(R), \forall
 \eta\in L^2_T(R), \langle q_T, \eta\rangle
 =\tilde{y}_0=\mathcal{E}_{0,T}^{\tilde{g}}(\eta)$,
 $\tilde{g}(t,y,z)=g_y(y_t,z_t)y+g_z(y_t,z_t)z$,
\begin{equation}\label{ciweifen}
\tilde{y}_t= \eta +  \int_t^T
\tilde{g}(s,\tilde{y}_s,\tilde{z}_s)ds-\int_t^T\tilde{z}^*_sdW_s, 0\leq t\leq T,
\end{equation}
 i.e., $\tilde{y}_t$ is the solution of BSDE with generator $\tilde{g}(t,y,z)$.
\end{lem}
\begin{proof}
By Ji and Peng \cite {sl-ji&sg-pe}, if $g$ is continuously differentiable in  $(y,z)\in \times R\times R^d$ with bounded derivatives, then
$\mathcal{E}_{0,T}^g(\cdot)$ is strict differentiable and
for any $\eta\in L_T^2(R)$,
$$\lim_{\zeta\rightarrow \xi,t\rightarrow 0^+}\frac{\mathcal{E}_{0,T}^g(\zeta+t\eta)-\mathcal{E}_{0,T}^g(\zeta)}{t}=\tilde{y}_0.$$
It is obviously that the function $\tilde{y}_0=\mathcal{E}_{0,T}^{\tilde{g}}(\eta)$ deduced by \eqref{ciweifen} is linear continuous on $L_T^2(R)$,
hence by the Riesz representation theorem, there exists  $q_T\in L_T^2(R)$ such that $\langle q_T,
\eta\rangle=\mathcal{E}_{0,T}^{\tilde{g}}(\eta)$ holds for any $ \eta \in
L_T^2(R)$
and the corresponding sub-differential set contains only one element
$q_T$.
\end{proof}
\begin{corollary}\label{necessary2}
Suppose  $g(t,y,z)$ has bounded continuous  derivatives in $(y,z)$  and
$(y_t=\mathcal{E}_{0,t}^g(\xi),z_t)$ is a solution of BSDE with terminal value $\xi$.
If $\xi^*$ is an optimal solution of \eqref{minmizerisk},
then there exists $\zeta\in \partial^o \rho(\xi^*)$  and some positive number $\lambda$
such that
$$\zeta+ \lambda q_T= 0,$$
 where $q_T$ is obtained by the Riesz representation theorem through \eqref{ciweifen}.
\end{corollary}

The key points for our successes to use non-smooth results are  Proposition \ref{bsdelip} and  Lemma \ref{0notbelong} and we can take $h(\cdot)$ as $\mathcal{E}_{0,T}^g(\cdot)$ in Proposition \ref{lemcalofnor}.  But in constrained case, the function $\mathcal{E}_{0,T}^{g,\phi}(\cdot)$ fails in both Proposition \eqref{bsdelip}  and Lemma \eqref{0notbelong}. In such case, we try to describe $\partial d^o_C(\xi)$ or $N_C(\xi)$ for $\xi\in C$ directly, where
$C:=\mathcal{A}^\phi(x):=\{\xi\in L_T^2(R)|\mathcal{E}_{t,T}^{g,\phi}(\xi)\leq x\}$. Thanks to the lower-semi continuity of $\mathcal{E}_{t,T}^{g,\phi}(\cdot)$,
the constrained set $C$ in our optimal problem  \eqref{cminmizerisk} is closed and many results about distance function $d_C(\cdot)$ of closed set $C$ thus can be applied here.

\section{Examples}
In this section, some examples are proposed to illustrate the obtained result.
In Ji and Peng \cite {sl-ji&sg-pe}, they
 considered the following optimal problem to find $\xi^*$ such that
\begin{equation}\label{optimizationji}
J(\xi):=E[u(\xi)]
\end{equation}
is minimized under the following constraints
\[
\left\{
\begin{array}{lll}
E[\varphi(\xi)]=c,\\
\mathcal{E}_{0,T}^{g}(\xi)=x,\\
\xi\in U,
\end{array}
\right.
\]
where
$$U=\{\xi|\xi\in L_T^2(\Omega), \xi\geq 0, \quad a.s.\}$$
and functions $u,\varphi$ are both continuous differentiable with bounded derivatives. In the framework of  Ji and Peng \cite
{sl-ji&sg-pe}, if we take the notations of non-smooth analysis, Ji and Peng \cite
{sl-ji&sg-pe} obtained the following result.
\begin{thm}\label{their}
If $\xi^*$ is an optimal solution of \eqref{optimizationji}, then there exist real number $
h^1$ and non-positive number $h$  such that
\begin{equation}
hu_x\xi^*(w))+h^1\varphi_x(\xi^*(w))+q_T\geq 0, \forall w\in M,
a.s.,
\end{equation}
\begin{equation}
hu_x(\xi^*(w))+h^1\varphi_x(\xi^*(w))+q_T(w)= 0, \forall w\in M^c,
a.s.,
\end{equation}
where $q_T$ is the terminal value of BSDE \eqref{joi-equ},
$M:=\{w|\xi^*(w)= 0\}$.
\end{thm}
 Since $E[\varphi(\xi)]=c$ is a constant,
then we can let $\rho(\xi):=E[u(\xi)+\varphi(\xi)]$ and transfer \eqref{optimizationji}
to our optimal problem \eqref{minmizerisk}.
For  $\rho(\xi):=E[u(\xi)+\varphi(\xi)]$, we have the following proposition.
\begin{pro}
If $u(x):R\rightarrow R$ is continuous differentiable with bounded derivatives,
then the function $f(\xi)=E[u(\xi)]$ defined on  $L_T^2(R)$ is absolutely differentiable,
and its sub-differential is $u_x(\xi)$.
\end{pro}
\begin{proof}
It can be proved by Fubini theorem and bounded convergence theorem.
\end{proof}
 Noting that when $u$ and $\varphi$ are both continuously differentiable, by the above Proposition, $\rho(\xi)$ is absolutely differentiable,
then there exists only one element  in $\partial \rho(\xi)$,
i.e., $u_x(\xi)+\varphi_x(\xi)$,
where $u_x(\cdot),\varphi_x(\cdot)$ is the corresponding derivative.
At the same time, by Lemma \ref{gen-gra-bsde},
when the generator $g$ of  BSDE is continuously differentiable,
the sub-differential of the function  $\mathcal{E}_{0,T}^{g}(\xi)$ deduced by BSDE contains only $q_T$,
then by Corollary  \ref{necessary2}, there exists a number $\lambda$, such that
\begin{equation}\label{our}
q_T+ \lambda (u_x(\xi)+\varphi_x(\xi)) =0.
\end{equation}
We consider a similar optimal investment problem by non-smooth analysis via BSDE approach.
\begin{example}\label{example1}\rm
Minimize
\begin{equation}
\min E[\xi^2]-c^2
\end{equation}
in a set of variables satisfying the following constraints
\[
\left\{
\begin{array}{lll}
E[\xi]=c,\\
\mathcal{E}_{0,T}^{g}(\xi)\leq x,\\
\xi\in L^2_T(R), \quad \xi\geq 0, \ \mbox{a.s.},
\end{array}
\right.
\]
where $g(t,y,z)=r(t)y+\theta(t)z$,
$r(t)$ and $\theta(t)$ are coefficients derived from financial market satisfying suitable measurable and integrable conditions.
\end{example}
In this example, since $E[\xi]=c$ is a constant, for any $b\in R$, we take $\rho(\xi):=E[\xi^2+b\xi]$ and it is obviously absolutely differentiable, then the sub-differential at  $\xi^*$ only contains $2\xi^*+b$.
It is easy to get $\partial\mathcal{E}_{0,T}^{g}(\xi^*)=\{q_T\}$,
where $(\tilde{y}_t, \tilde{y}_t)$ is the solution of following BSDE
$$\tilde{y}_t=\xi^*+\int_t^T
(r(s)\tilde{y}_s+\theta(s)\tilde{z}_s)ds-\int_t^T\tilde{z}^*_sdW_s, 0\leq t \leq T.$$
Then, if $\xi^*$ is an optimal solution of this example, then by
Theorem \ref{necessary1} or Corollary  \ref{necessary2},
there exists a number $\lambda_b$ such that
$$2\xi^*+b+\lambda_b q_T=0$$
holds.

The virtue of our method can help us  consider the optimal problem when the expectation is not lower than a level.
\begin{example}\label{example2}\rm
Finding an optimal $\xi^*$ in the following set
\[
\left\{
\begin{array}{lll}
E[\xi]\geq c,\\
\mathcal{E}_{0,T}^{g}(\xi)\leq x,\\
\xi\in L^2_T(R), \quad \xi\geq 0, \ \mbox{a.s.},
\end{array}
\right.
\]
to minimize
\begin{equation}
\min E[\xi^2]-E^2[\xi].
\end{equation}
\end{example}
Because the constraint on the expectation is not a constant,
we take $\rho(\xi):=E[\xi^2]-E^2[\xi]$. We combine  the constraints on the expectation and  initial value of investment together to get the following new constraint
$$\mathcal{A}:=\{\xi\in L_T^2(R)|h(\xi)\leq 0\},$$
where $h(\xi):=\max\{\mathcal{E}_{0,T}^{g,\phi}(\xi)-x, -E[\xi]+c\}$.
\begin{lem} {\rm(Clark \cite{fh-cl})}
Supposing that $\{f_i, i=1,2,\cdots n\}$ is a set of  Lipschitz  functions, we define
$$f(x):=\max\{f_i(x)|i=1,2,\cdots n\}.$$
Let $I(x)$ be the subset of index satisfying $f_i(x)=f(x)$, then
\begin{equation}
\partial f(x)\subset co\{\partial f_i(x):i\in I(x)\}.
\end{equation}
Furthermore, if $f_i$ is normal, then the equality holds,
where $coA$ is the convex hull of  $A$.
\end{lem}
By the Lemma stated above and Theorem \ref{necessary1}, we have the following theorem.
\begin{thm}
 If $\xi^*$ is an optimal solution of Example  \ref{example2}, then there exist a nonnegative number  $\lambda$ and $a\in [0,1]$ such that

\[
\left\{
\begin{array}{lll}
2(\xi^*+E[\xi^*])+\lambda q_T=0, &\text{当}\mathcal{E}_{0,T}^g(\xi^*)-x>-E[\xi^*]+c,\\
2(\xi^*+E[\xi^*])-\lambda=0, &\text{当}\mathcal{E}_{0,T}^g(\xi^*)-x<-E[\xi^*]+c,\\
2(\xi^*+E[\xi^*])+\lambda((1-a)q_T-a)=0,
&\text{当}\mathcal{E}_{0,T}^g(\xi^*)-x=-E[\xi^*]+c.
\end{array}
\right.
\]
\end{thm}
In the classic investment problem, one often take variance as a risk measure,
a mean-variance method is used in many literatures. But by Delbaen
\cite{f-de} or F\"{o}llmer and Schied \cite{h-fo&a-sc1},
such kind of risk measure is not perfect. We often take $\rho(\cdot)$ as a coherent or convex risk measure in
\eqref{minmizerisk}.
In Gianin \cite{er-gi}, when $g$  is a sub-additive homogeneous function satisfying  some usual  conditions,
 we can define a risk measure via $\rho(\xi):=\mathcal{E}_{0,T}^g(-\xi)$.
\begin{example}\rm\label{example3}
Suppose $f$ is sub-additive homogeneous function satisfying usual  conditions   and independent of $y$,
$f_z$ is continuously bounded,
define $\rho(\xi):=\mathcal{E}_{0,T}^f(-\xi)$, we want to find an optimal $\xi^*$ element  in the following constrained set
\[
\left\{
\begin{array}{lll}
E[\xi]\geq c,\\
\mathcal{E}_{0,T}^{g}(\xi)\leq x,\\
\xi\in L^2_T(R), \quad \xi\geq 0, \mbox{a.s.},
\end{array}
\right.
\]
to minimize
\begin{equation}
\min \rho(\xi).
\end{equation}
\end{example}
If $\xi^*$ is an optimal solution in this example, then we can obtain similar conditions like above examples with new sub-differential sets.
By the assumptions of $f$, we can see obviously
$\partial\rho(\xi)$ contains only one element  $\bar{y}_T$,
where $(\bar{y}_t,\bar{z}_t)$ is the solution of following BSDE
$$\bar{y}_t= \xi^* +  \int_t^T f_z(s,y_s^*,z^*_s)\bar{z}_sds-\int_t^T\bar{z}^*_sdW_s, 0\leq t\leq T  $$
and $(y^*_t,z^*_t)$ is the solution of BSDE generated by $g(t,y,z)$ with terminal value $\xi^*$.
Thus we have the following result.
\begin{thm}
 If $\xi^*$ is an solution of Example  \ref{example3}, then there exist a nonnegative number $\lambda$ and $a\in [0,1]$  such that
\[
\left\{
\begin{array}{lll}
2\bar{y}_T+\lambda q_T=0, &\text{当}\mathcal{E}_{0,T}^g(\xi^*)-x>-E[\xi^*]+c,\\
2\bar{y}_T-\lambda=0, &\text{当}\mathcal{E}_{0,T}^g(\xi^*)-x<-E[\xi^*]+c,\\
2\bar{y}_T+\lambda((1-a)q_T-a)=0,
&\text{当}\mathcal{E}_{0,T}^g(\xi^*)-x=-E[\xi^*]+c.
\end{array}
\right.
\]
\end{thm}

\section{Appendix:some results about non-smooth analysis}
Suppose that $X$ is a Banach space, $X^*$ is its dual space. A function  $f:X\rightarrow R $ is called
Lipschitzian if
\begin{equation}
|f(x_1)-f(x_2)|\leq M||x_1-x_2||
\end{equation}
holds for some $M>0$,
where $||\cdot ||$ is the norm in  $X$.

The generalized directional derivative of $f$ at $x$, denoted as $f^o(x;v)$, is defined as
\begin{equation}
f^o(x;v):=\limsup_{y\rightarrow x,t\downarrow 0}\frac{f(y+tv)-f(y)}{t},
\end{equation}
where $y$ is a vector in  $X$, $t$ is a positive number.

Obviously, $f^o(x;v)$ is  homogeneous and sub-linear  on $X$, then by Banach Theorem, the generalized derivative set of $f$ at $x$
\begin{equation}
\partial^o f(x):= \{\zeta\in X^*| \zeta(v)\leq  f^o(x;v), \forall v\in X\}
\end{equation}
is nonempty and weak star compact in $X^*$.

By definition, the Fermat optimal principle
$0\in \partial^o f(x_0) $ holds when $f(x)$ attains extreme at some point  $x_0\in X$,
$$f(x)\geq f(x_0),\qquad \forall x\in X.$$
Now, we recall more results in non-smooth analysis. For more details, one can see Clark et al.
\cite{fh-cl&ys-le&rj-st&pr-wo}.

Given a set $C\subset X$,  the distance function $d_C(x):X\rightarrow R$ is defined as
$$d_C(x):=\inf\{||y-x||,y\in C\}.$$
The following lemma transfers the constrained problem to the unconstrained case.
\begin{lem}{\rm (Exact penalization)}\label{exactpenalization}
Suppose that $f$ is a Lipschitz function with coefficient $K$ defined on  $S$,  $x\in C\subset S$ and   $f$ takes its minimum value at  $x$  on $C$. Then, for any $\hat{K}\geq K$,
  $g(y)=f(y)+\hat{K} d_C(y)$ attains minimum value at    $x$ on $S$. On the contrary, if  $\hat{K}>K$ and $C$ is closed, then the minimum point of $g$ on $S$
 must belong to $C$.
\end{lem}
Contingent and normal derivative for a set $C$ are defined by the distance function,  see Clark et al.
 \cite{fh-cl&ys-le&rj-st&pr-wo} for details.
\begin{definition}
  Assume $x\in C$, if $d^o_C(x;v)= 0$, then $v$ is said to be a contingent derivative at $x\in X$. We denote the set of contingent derivatives as $T_C(x)$. By
  polarity, we define the normal  derivative set as
$$N_C(x):=\{\zeta\in X^*| \zeta(v)\leq 0, \quad \forall v\in T_C(x)\}.$$
\end{definition}
By the above definition, we have the following proposition.
\begin{pro}
Supposing  that $x\in C$, then it holds that
\begin{equation}\label{repofnor}
N_C(x)= cl\left\{\bigcup_{\lambda\geq 0}\lambda \partial^o d_C(x)\right\},
\end{equation}
where $cl$ means the weak star closure.\end{pro}

For  $C=\mathcal{A}(x)$, if
$\xi^*$ is an optimal solution of \eqref{minmizerisk}, then the Fermat condition
\begin{equation}\label{necessaryconditon2}
0\in \partial^o \rho(\xi^*)+ N_C(\xi^*).
\end{equation}
holds.

For a special kind of set $C$, we have the following result.
\begin{pro}\label{lemcalofnor}
Suppose that $h$ is Lipschitz in a neighborhood of $x$ and $0\notin \partial^o h(x)$, if  $C=\{y\in X: h(y)\leq h(x)\}$,
then it holds that
\begin{equation}\label{calofnor}
N_C(x)\subset \bigcup_{\lambda\geq 0}\lambda\partial^o h(x).
\end{equation}
\end{pro}


\begin{thebibliography}{}

\bibitem{tr-bi&h-ji&sr-pl&xy-zh}T.R. Bielecki, H. Jin, S.R. Pliska, X.Y. Zhou, Continuous time mean variance portfolio selection with bankruptcy prohibition, Mathematical  Finance 15 (2005) 213--244

\bibitem{fh-cl}F.H. Clark, Optimization and nonsmooth analysis, John
Wiley and Sons, Inc., New York. (1983)


\bibitem{fh-cl&ys-le&rj-st&pr-wo} F.H. Clark, Y.S. Ledyaev, R.J. Stern, P.R. Wolenski,
Nonsmooth Analysis and Control Theory, Springer. (1998)

\bibitem{f-de}F. Delbaen, Coherent risk measures on general probability space,
Advance in Finance and Stochastics, springer-verlag, (2002) 1--37


 \bibitem{h-fo&a-sc1} H. F\"{o}llmer, A. Schied, Stochastic Finance: An Introduction in Discrete Time, De Gruyter, Berlin, New York. (2004)

\bibitem{er-gi}E.R. Gianin, Risk measures via
$g$-expectations, Insurance: Mathematics and Economics 39 (2006) 19--34
\bibitem{e-pa&sg-pe}E. Pardoux, S.G. Peng: Adapted solution of a backward stochastic differential equation. Systems Control Letters, 14, 55--62 (1990).

\bibitem{sl-ji&sg-pe}S.L. Ji, S.G. Peng, Terminal perturbation method for the backward approach to continuous time mean-variance portfolio selection, Stochastic Processes and their Applications, 118 (2008) 952--967

\bibitem{ne-k&sg-p&mc-q}N.El. Karoui,  S.G. Peng,  M.C. Quenez, Backward stochastic differential equations in finance. Mathematical Finance, 7(1), (1997) 1--71



\bibitem{hl-wr&y-ren} H.L. Wu, Y. Ren, Continuous dependence property of Constrained BSDE, revision under review


\end{thebibliography}
\end{document}